\documentclass[12pt]{amsart}
\usepackage[cp1251]{inputenc}
\usepackage[T2A]{fontenc}
\usepackage[english]{babel}
\usepackage{amsmath,amsfonts,amssymb}
\usepackage{geometry}
\usepackage{amsmath, amsthm, amscd, amsfonts, amssymb, graphicx, color}
\usepackage[bookmarksnumbered, colorlinks, plainpages]{hyperref}

\newtheorem{theorem}{Theorem}

\newtheorem{proposition}{Proposition}

\theoremstyle{remark}
\newtheorem{remark}{Remark}

\theoremstyle{definition}

\begin{document}

\title[Elliptic problems in the sense of B.~Lawruk on refined scales]{Elliptic problems in the sense of B.~Lawruk\\on two-sided refined scales of spaces}


\author[I. Chepurukhina]{Iryna S. Chepurukhina}

\address{Institute of Mathematics, National Academy of Sciences of Ukraine,
3 Tereshchenkivs'ka, Kyiv, 01601, Ukraine}

\email{Chepuruhina@mail.ru}


\author[A. Murach]{Aleksandr A. Murach}

\address{Institute of Mathematics, National Academy of Sciences of Ukraine,
3 Tereshchenkivs'ka, Kyiv, 01601, Ukraine}

\email{murach@imath.kiev.ua}


\subjclass[2010]{Primary 35J40, 46E35}

\dedicatory{To the blessed memory of Professor V.~V.~Sharko}


\keywords{Elliptic boundary-value problem, slowly varying function, H\"ormander space, two-sided refined scale, Fredholm operator, a~priori estimate for solutions, local regularity of solutions}

\begin{abstract}
We investigate elliptic boundary-value problems with additional unknown functions on the boundary of a Euclidean domain. These problems were introduced by Lawruk. We prove that the operator corresponding to such a problem is bounded and Fredholm on two-sided refined scales built on the base of the isotropic H\"ormander inner product spaces. The regularity of the distributions forming these spaces are characterized by a real number and an arbitrary function that varies slowly at infinity in the sense of Karamata. For the generalized solutions to the problem, we prove theorems on a priori estimates and local regularity in these scales. As applications, we find new sufficient conditions under which the solutions have continuous classical derivatives of a prescribed order.
\end{abstract}

\maketitle

\section{Introduction}\label{sec1}

The paper is devoted to the investigation of the operators generated by
elliptic bound\-ary-value problems with additional unknown functions on the boundary of a Euclidean domain. Such problems were introduced by B.~Lawruk \cite{Lawruk63a, Lawruk63b, Lawruk65}. They appear naturally if we pass from a general (nonregular) elliptic  boundary-value problem to its formally adjoint problem. Moreover, the class of these problems are closed with respect to this passage. As to applications, various problems in hydrodynamics and the theory of elasticity proved to belong to this class \cite{AslanyanVassilievLidskii81, Garlet90, NazarovPileckas93}.

This class has been investigated completely enough in the Sobolev spaces; see the monographs by V.~A. Kozlov, V.~G. Maz'ya, J.~Rossmann \cite[Chapt.~3]{KozlovMazyaRossmann97} and Ya.~Roitberg \cite[Chapt.~2]{Roitberg99}. Among the proved results are theorems on the Fredholm property of the operators corresponding to the problem under investigation, on the  isomorphisms generated by these operators, on a priori estimates for solutions to the problem, theorems on the increase in local regularity of these solutions. These results have been established for the two-sided scales (of function spaces) that modify the classical Sobolev scale in the sense of Ya.~Roitberg \cite{Roitberg64, Roitberg65} (see also his monograph \cite[Sect.~2]{Roitberg96}, were this modification is systematically used in the theory of usual elliptic boundary-value problems).

The purpose of this paper is to prove versions of these theorems for two-sided refined scales built on the base of the isotropic H\"ormander inner product spaces
$$
H^{s,\varphi}(\mathbb{R}^{n}):=
\bigl\{w\in\mathcal{S}'(\mathbb{R}^{n}):\,
\langle\xi\rangle^{s}\varphi(\langle\xi\rangle)\widehat{w}(\xi)\in
L_{2}(\mathbb{R}^{n},d\xi)\bigr\},
$$
parametrized with the number $s\in\mathbb{R}$ and the function $\varphi:[1,\infty)\rightarrow(0,\infty)$ that varies slowly at infinity in the sense of J.~Karamata. Here, $\widehat{w}$ is the Fourier transform of the tempered distribution $w$, whereas $\langle\xi\rangle:=(1+|\xi|^{2})^{1/2}$. These spaces form the refined Sobolev scale, which was selected and investigated by V.~A.~Mikhailets and A.~A.~Murach \cite{MikhailetsMurach05UMJ5, MikhailetsMurach06UMJ3,  MikhailetsMurach08MFAT1}. This scale has an important interpolation property; namely, each space $H^{s,\varphi}(\mathbb{R}^{n})$ is a result of the interpolation with an appropriate function parameter of the Sobolev inner product spaces $H^{s-\varepsilon}(\mathbb{R}^{n})$ and $H^{s+\delta}(\mathbb{R}^{n})$ with $\varepsilon,\delta>0$.

V.~A.~Mikhailets and A.~A.~Murach built the theory of solvability of general elliptic boundary-value problems in the one-sided refined Sobolev scale and its two-sided modifications [16--20, 22--24]. The mentioned interpolation is a key method in their theory. In this connection, we also note papers \cite{AnopMurach14MFAT2, AnopMurach14UMJ7, Slenzak74}. However, this theory does not involve the important class of elliptic boundary-value problems in the sense of B.~Lawruk.

This paper consists of seven sections. Section~\ref{sec1} is Introduction. In Section~\ref{sec2}, we state an elliptic boundary-value problem in the sense of B.~Lawruk and consider the corresponding Green formula and formally adjoint problem. In Section~\ref{sec3}, we give the definitions of function spaces that form the two-sided refined scales. The main results of the paper are formulated in Section~\ref{sec4}. They are the theorems on the character of solvability of the problem under investigation and on properties of its generalized solutions. Section~\ref{sec5} is devoted to the method of interpolation with a function parameter of Hilbert spaces and its applications to the refined scales. The main results are proved in Section~\ref{sec6}. In Section~\ref{sec7}, we give some applications of the refined scales to the investigation of classical smoothness of the generalized solutions.

\section{Statement of the problem}\label{sec2}

Let $\Omega$ be a bounded domain in the Euclidean space $\mathbb{R}^{n}$, with $n\geq2$. Suppose that the boundary $\Gamma:=\partial\Omega$ of $\Omega$ is an infinitely smooth closed manifold of dimension $n-1$. As usual, $\overline{\Omega}:=\Omega\cup\Gamma$. Let $\nu(x)$ denote the unit vector of the inner normal to $\Gamma$ at a point $x\in\Gamma$.

We arbitrarily choose integers $q\geq1$, $\varkappa\geq1$, $m_{1},\ldots,m_{q+\varkappa}\in[0,2q-1]$ and $r_{1},\ldots,r_{\varkappa}$.
We consider the linear boundary-value problem in the domain $\Omega$ with $\varkappa$ additional unknown functions on the boundary $\Gamma$:
\begin{gather}\label{2f1}
Au=f\quad\mbox{in}\quad\Omega,\\
B_{j}u+\sum_{k=1}^{\varkappa}C_{j,k}v_{k}=g_{j}\quad\mbox{on}\quad\Gamma,
\quad j=1,...,q+\varkappa.\label{2f2}
\end{gather}
Here,
$$
A:=A(x,D):=\sum_{|\mu|\leq 2q}a_{\mu}(x)D^{\mu}
$$
is a linear differential operator on $\overline{\Omega}$ of the even order $2q$. Moreover, every
$$
B_{j}(x,D):=\sum_{|\mu|\leq m_{j}}b_{j,\mu}(x)D^{\mu}
$$
is a linear boundary differential operator on $\Gamma$ of order $m_{j}$, and each $C_{j,k}:=C_{j,k}(x,D_{\tau})$ is a linear tangent differential operator on $\Gamma$ with $\mathrm{ord}\,C_{j,k}\leq m_{j}+r_{k}$. All the coefficients of these differential operators are infinitely smooth complex-valued functions given on $\overline{\Omega}$ and $\Gamma$ respectively.

We use the following standard notation:
$\mu:=(\mu_{1},\ldots,\mu_{n})$ is a multi-index with $|\mu|:=\mu_{1}+\ldots+\mu_{n}$, and
$D^{\mu}:=D_{1}^{\mu_{1}}\ldots D_{n}^{\mu_{n}}$ with  $D_{\ell}:=i\partial/\partial x_{\ell}$, where $i$ is imaginary unit and  $x=(x_1,\ldots,x_n)$ is an arbitrary point in $\mathbb{R}^{n}$. Moreover, we put $D_{\nu}:=i\partial/\partial\nu(x)$ and  $\xi^{\mu}:=\xi_{1}^{\mu_{1}}\ldots\xi_{n}^{\mu_{n}}$ for  $\xi=(\xi_{1},\ldots\xi_{n})\in\mathbb{C}^{n}$.

The function $u$ given on $\Omega$ and all the functions  $v_{1},\ldots,v_{\varkappa}$ given on $\Gamma$ are unknown in
the boundary-value problem \eqref{2f1}, \eqref{2f2}. Note that all functions and distributions are assumed to be complex-valued in the paper.

We suppose that the boundary-value problem \eqref{2f1}, \eqref{2f2} is elliptic  in $\Omega$ in the sense of B.~Lawruk \cite{Lawruk63a}. We recall the corresponding definition (see also \cite[Subsect.~3.1.3]{KozlovMazyaRossmann97}).

Let $A^{(0)}(x,\xi)$, $B_{j}^{(0)}(x,\xi)$, and $C_{j,k}^{(0)}(x,\tau)$
denote the principal symbols of the differential operators $A(x,D)$, $B_{j}(x,D)$, and $C_{j,k}(x,D_{\tau})$ respectively. Recall that
$$
A^{(0)}(x,\xi):=\sum_{|\mu|=2q}a_{\mu}(x)\xi^{\mu},\quad\mbox{with}\;\; x\in\overline{\Omega}\;\;\mbox{and}\;\;\xi\in\mathbb{C}^{n},
$$
is a homogeneous polynomial of order $2q$ in $\xi$ and that
$$
B_{j}^{(0)}(x,\xi):=\sum_{|\mu|=m_{j}}b_{j,\mu}(x)\xi^{\mu},\quad\mbox{with}\;\; x\in\Gamma\;\;\mbox{and}\;\;\xi\in\mathbb{C}^{n},
$$
is a homogeneous polynomial of order $m_{j}$ in $\xi$. Moreover, for each point $x\in\Gamma$, the expression $C^{(0)}_{j,k}(x,\tau)$ is a homogeneous polynomial of order $m_{j}+r_{k}$ in $\tau$, where  $\tau$ is a tangent vector to the boundary $\Gamma$ at the point $x$. (If $\mathrm{ord}\,C_{j,k}<m_{j}+r_{k}$, then $C^{(0)}_{j,k}(x,\tau):=\nobreak0$.)

The boundary-value problem \eqref{2f1}, \eqref{2f2} is called elliptic in the domain $\Omega$ if the following three conditions are satisfied:
\begin{itemize}
\item [(i)] The differential operator $A(x,D)$ is elliptic at each point $x\in\overline{\Omega}$; i.e., $A^{(0)}(x,\xi)\neq\nobreak0$ for an arbitrary vector $\xi\in\mathbb{R}^{n}\setminus\{0\}$.
\item [(ii)] The differential operator $A(x,D)$ is properly elliptic at each point $x\in\Gamma$; i.e., for an arbitrary tangent vector $\tau\neq0$ to $\Gamma$ at $x$, the polynomial   $A^{(0)}(x,\tau+\zeta\nu(x))$ in $\zeta\in\mathbb{C}$ has $q$ roots with positive imaginary part and $q$ roots with negative  imaginary part (of course, these roots are calculated with regard for their multiplicity ).
\item [(iii)] The system of boundary-value conditions \eqref{2f2} covers equation \eqref{2f1} at each point $x\in\Gamma$. This means that, for an arbitrary vector $\tau\neq0$ from condition (ii), the boundary-value problem
    \begin{gather*}
    A^{(0)}(x,\tau+D_{t}\,\nu(x))\theta(t)=0\quad\mbox{for}\quad t>0,\\
    B_{j}^{(0)}(x,\tau+D_{t}\,\nu(x))\theta(t)\big|_{t=0}+
    \sum_{k=1}^{\varkappa}C_{j,k}^{(0)}(x,\tau)\lambda_{k}=0,\\ j=1,...,q+\varkappa,
    \end{gather*}
    has only the trivial (zero) solution. This problem is posed relative to the unknown function $\theta\in C^{\infty}([0,\infty))$ with $\theta(t)\rightarrow0$ as $t\rightarrow\infty$ and the unknown complex-valued numbers $\lambda_{1},\ldots,\lambda_{\varkappa}$. Here, $A^{(0)}(x,\tau+D_{t}\,\nu(x))$ and $B_{j}^{(0)}(x,\tau+D_{t}\,\nu(x))$ are differential operators with respect to $D_{t}:=i\partial/\partial t$. We obtain them if we put $\zeta:=D_{t}$ in the polynomials $A^{(0)}(x,\tau+\zeta\nu(x))$ and $B_{j}^{(0)}(x,\tau+\zeta\nu(x))$, in $\zeta$, respectively.
\end{itemize}

Note that condition (ii) follows from condition (i) in the case of $n\geq3$.

Various examples of the elliptic boundary-value problem \eqref{2f1}, \eqref{2f2} are given in \cite[Subsect.~3.1.5]{KozlovMazyaRossmann97}. Specifically, the boundary-value problem
\begin{equation*}
\Delta u=f\;\;\mbox{in}\;\;\Omega,\quad
u+v=g_{1}\;\;\mbox{and}\;\;D_{\nu}u+D_{\tau}v=g_{2}\;\;\mbox{on}\;\;\Gamma
\end{equation*}
is elliptic in $\Omega$. Here, $n=2$, $\varkappa=1$, $\Delta$ is the Laplace operator, and  $D_{\tau}:=i\partial/\partial\tau$, where $\partial/\partial\tau$ is the derivative along the curve $\Gamma$.

With the problem \eqref{2f1}, \eqref{2f2}, we connect the linear mapping
\begin{equation}\label{2f3}
\begin{gathered}
\Lambda:(u,v_{1},...,v_{\varkappa})\mapsto
\biggl(Au,\,B_{1}u+\sum_{k=1}^{\varkappa}C_{1,k}v_{k},...,
B_{q+\varkappa}u+\sum_{k=1}^{\varkappa}C_{q+\varkappa,k}v_{k}\biggr),\\
\mbox{with}\quad u\in C^{\infty}(\overline{\Omega})\quad\mbox{and}\quad
v_{1},\ldots,v_{\varkappa}\in C^{\infty}(\Gamma).
\end{gathered}
\end{equation}
We will investigate properties of the extension (by continuity) of this mapping in appropriate couples of Hilbert spaces that form certain two-sided refined scales.

To describe the range of this extension, we need the following Green formula (see \cite[Theorem~3.1.1]{KozlovMazyaRossmann97}:
\begin{equation*}
\begin{gathered}
(Au,w)_{\Omega}+\sum_{j=1}^{q+\varkappa}\biggl(B_{j}u+
\sum_{k=1}^{\varkappa}C_{j,k}v_{k},h_{j}\biggr)_{\Gamma}=\\
=(u,A^{+}w)_{\Omega}+\sum_{j=1}^{2q}\biggl(D_{\nu}^{j-1}u,K_{j}w+
\sum_{k=1}^{q+\varkappa}Q_{k,j}^{+}h_{k}\biggr)_{\Gamma}+\sum_{j=1}^{\varkappa}\biggl(v_{j},
\sum_{k=1}^{q+\varkappa}C_{k,j}^{+}h_{k}\biggr)_{\Gamma},
\end{gathered}
\end{equation*}
where $u,w\in C^{\infty}(\overline{\Omega})$, $v=(v_{1},\ldots,v_{\varkappa})\in(C^{\infty}(\Gamma))^{\varkappa}$,  $h=(h_{1},\ldots,h_{q+\varkappa})\in(C^{\infty}(\Gamma))^{q+\varkappa}$ are arbitrary, whereas $(\cdot,\cdot)_{\Omega}$ and $(\cdot,\cdot)_{\Gamma}$ are the inner products in the Hilbert spaces $L_{2}(\Omega)$ and $L_{2}(\Gamma)$ of functions square integrable over $\Omega$ and $\Gamma$ respectively.

Here, $A^{+}$ denotes the differential operator that is formally adjoint to $A$ with respect to $(\cdot,\cdot)_{\Omega}$. Moreover, $C_{k,j}^{+}$ and $Q_{k,j}^{+}$ denote the tangent differential operators that are formally adjoint to $C_{k,j}$ and $Q_{k,j}$ respectively relative to $(\cdot,\cdot)_{\Gamma}$, the tangent differential operators $Q_{k,j}$
appearing in the representation of the boundary differential operators $B_{j}$ in  the form
$$
B_{j}(x,D)=
\sum_{k=1}^{2q}Q_{j,k}(x,D_{\tau})D_{\nu}^{k-1}.
$$
(Of course, if $k>m_{j}$, then $Q_{j,k}=0$). Finally,
$$
K_{j}:=K_{j}(x,D):=\sum_{k=1}^{2q-j+1}T_{j,k}(x,D_{\tau})D_{\nu}^{k-1},
$$
where every $T_{j,k}(D_{\tau})$ is a certain tangent differential operator on $\Gamma$ of the order $\mathrm{ord}\,T_{j,k}\leq2q+1-j-k$.

In view of the Green formula, we consider the following boundary-value problem in $\Omega$ with $q+\varkappa$ additional unknown function on $\Gamma$:
\begin{gather}\label{2f4}
A^{+}w=\omega\quad\mbox{in}\quad\Omega,\\
\label{2f5}
K_{j}w+\sum_{k=1}^{q+\varkappa}Q_{k,j}^{+}h_{k}=\chi_{j}\quad
\mbox{on}\quad\Gamma,\quad j=1,...,2q,\\
\label{2f6}
\sum_{k=1}^{q+\varkappa}C_{k,j}^{+}h_{k}=\chi_{2q+j}
\quad\mbox{on}\quad\Gamma,\quad j=1,...,\varkappa.
\end{gather}
This problem is formally adjoint to the problem \eqref{2f1}, \eqref{2f2} with respect to the Green formula. Note that the problem \eqref{2f1}, \eqref{2f2} is elliptic if and only if the problem \eqref{2f4}, \eqref{2f5}, \eqref{2f6} is elliptic (see \cite[Theorem~3.1.2]{KozlovMazyaRossmann97}).

\section{Refined scales of spaces}\label{sec3}

We will investigate the properties of the elliptic boundary-value problem \eqref{2f1}, \eqref{2f2} in appropriate couples of Hilbert spaces that form the two-sided refined scale introduced in \cite{MikhailetsMurach08UMJ4}. This scale is constructed on the base of some isotropic H\"ormander inner product spaces. In this section, we give the definitions of these spaces and the scale and, moreover, discuss some of their properties. We follow \cite{MikhailetsMurach08UMJ4} (see also \cite[Subsect. 4.2.2]{MikhailetsMurach10, MikhailetsMurach14}).

We start with the definition of the H\"ormander spaces $H^{s,\varphi}(\mathbb{R}^{n})$, with $s\in\mathbb{R}$ and $\varphi\in\mathcal{M}$. Here and below, $\mathcal{M}$ denotes the set of all Borel measurable functions $\varphi:[1,\infty)\rightarrow(0,\infty)$ such that both the functions $\varphi$ and $1/\varphi$ are bounded on each compact interval $[1,b]$, with $1<b<\infty$, and that the function $\varphi$ is slowly varying at infinity in the sense of J.~Karamata \cite{Karamata30a}, i.e., $\varphi(\lambda t)/\varphi(t)\rightarrow 1$ as $t\rightarrow\infty$ for each  $\lambda>0$.

Note that the slowly varying functions are well investigated and have various applications \cite{BinghamGoldieTeugels89, Seneta76}. The standard example of such functions is
$$
\varphi(t):=(\log t)^{r_{1}}(\log\log
t)^{r_{2}}\ldots(\underbrace{\log\ldots\log}_{k\;\mbox{\small times}}
t)^{r_{k}}\quad\mbox{of}\quad t\gg1,
$$
where $k\in\mathbb{Z}$, with $k\geq1$, and $r_{1},\ldots,r_{k}\in\mathbb{R}$ are arbitrary parameters.

Let $s\in\mathbb{R}$ and $\varphi\in\mathcal {M}$. By definition, the complex linear space $H^{s,\varphi}(\mathbb{R}^{n})$, with $n\geq1$, consists of all tempered distributions $w$ on $\mathbb{R}^{n}$ that the Fourier transform $\widehat{w}$ of $w$ is locally Lebesgue integrable over $\mathbb{R}^{n}$ and satisfies the condition
$$
\int\langle\xi\rangle^{2s}\varphi^{2}(\langle\xi\rangle)\,
|\widehat{w}(\xi)|^{2}\,d\xi<\infty.
$$
Here, the integral is taken over $\mathbb{R}^{n}$, and
$\langle\xi\rangle:=(1+|\xi|^{2})^{1/2}$. The inner product in this space is defined by the formula
$$
(w_{1},w_{2})_{H^{s,\varphi}(\mathbb {R}^{n})}:=
\int\langle\xi\rangle^{2s}\varphi^{2}(\langle\xi\rangle)
\widehat{w_{1}}(\xi)\overline{\widehat{w_{2}}(\xi)}d\xi
$$
and induces the norm $\|w\|_{H^{s,\varphi}(\mathbb{R}^{n})}:=
(w,w)_{H^{s,\varphi}(\mathbb{R}^{n})}^{1/2}$.

The space $H^{s,\varphi}(\mathbb{R}^{n})$ is a special isotropic Hilbert case of the spaces $\mathcal{B}_{p,\mu}$ introduced and investigated by L.~H\"ormander \cite[Sect.~2.2]{Hermander63} (see also \cite[Sect.~10.1]{Hermander83}). Namely, $H^{s,\varphi}(\mathbb{R}^{n})=\mathcal{B}_{p,\mu}$ provided that $p=2$ and $\mu(\xi)=\langle\xi\rangle^{s}\varphi(\langle\xi\rangle)$ for all $\xi\in\mathbb{R}^{n}$. The Hilbert spaces $\mathcal{B}_{2,\mu}$ were also investigated by L.~R.~Volevich and B.~P.~Paneah \cite[\S~2]{VolevichPaneah65}.

In the special case where $\varphi\equiv1$, the space $H^{s,\varphi}(\mathbb {R}^{n})$ becomes the Sobolev inner product space $H^{s}(\mathbb {R}^{n})$ of order $s\in\mathbb{R}$. Generally, we have the continuous and dense embeddings
\begin{equation}\label{2f7}
H^{s+\varepsilon}(\mathbb{R}^{n})\hookrightarrow H^{s,\varphi}(\mathbb{R}^{n})\hookrightarrow H^{s-\varepsilon}(\mathbb{R}^{n})
\quad\mbox{for every}\quad\varepsilon>0.
\end{equation}
They mean that, in the class of separable Hilbert spaces
$\{H^{s,\varphi}(\mathbb{R}^{n}):s\in\mathbb{R},\varphi\in\mathcal{M}\}$,
the function parameter $\varphi$ refines the main regularity (of distributions) characterized by the number $s$. This class is selected by V.~A.~Mikhailets and A.~A.~Murach \cite{MikhailetsMurach05UMJ5} and is called the refined Sobolev scale over $\mathbb{R}^{n}$ (see  \cite[Subsect.~1.3.3]{MikhailetsMurach14}).

We need analogs of the space $H^{s,\varphi}(\mathbb{R}^{n})$ for $\Omega$ and $\Gamma$. They are constructed in the standard way. Let us give the corresponding definitions \cite[Sect.~2.1 and 3.1]{MikhailetsMurach10, MikhailetsMurach14}.

By definition, the linear space $H^{s,\varphi}(\Omega)$ consists of
the restrictions $u=w\!\upharpoonright\!\Omega$ of all distributions $w\in H^{s,\varphi}(\mathbb{R}^{n})$ to $\Omega$. It is endowed with the norm
$$
\|u\|_{H^{s,\varphi}(\Omega)}:=
\inf\,\bigl\{\,\|w\|_{H^{s,\varphi}(\mathbb{R}^{n})}:\,
w\in H^{s,\varphi}(\mathbb{R}^{n}),\;\,
u=w\!\upharpoonright\!\Omega\bigr\}.
$$
The space $H^{s,\varphi}(\Omega)$ is Hilbert and separable with respect to this norm. The set $C^{\infty}(\overline{\Omega})$ is dense in this space.

The space $H^{s,\varphi}(\Gamma)$ consists of all distributions on $\Gamma$ that belong, in local coordinates, to $H^{s,\varphi}(\mathbb{R}^{n-1})$. Let us give the detailed definition. From $C^{\infty}$-structure on $\Gamma$, we arbitrarily choose a finite collection of local charts $\alpha_j: \mathbb{R}^{n-1}\leftrightarrow\Gamma_{j}$, with $j=1,\ldots,\lambda$, that the open sets $\Gamma_{1},\ldots,\Gamma_{\lambda}$ form a covering of the manifold $\Gamma$. Moreover, we arbitrarily choose functions $\chi_j\in C^{\infty}(\Gamma)$, with $j=1,\ldots,\lambda$, that form a partition of unity on $\Gamma$ which satisfies the condition $\mathrm{supp}\,\chi_j\subset\Gamma_j$. Then, by definition, the linear space $H^{s,\varphi}(\Gamma)$ consists of all distributions $h$ on $\Gamma$ such that $(\chi_{j}h)\circ\alpha_{j}\in H^{s,\varphi}(\mathbb{R}^{n-1})$ for every $j\in\{1,\ldots,\lambda\}$.
Here, $(\chi_{j}h)\circ\alpha_{j}$ is the representation of $\chi_{j}h$ in the local chart $\alpha_{j}$. This space is endowed with the inner product
$$
(h_{1},h_{2})_{H^{s,\varphi}(\Gamma)}:=
\sum_{j=1}^{\lambda}\,((\chi_{j}h_{1})\circ\alpha_{j},
(\chi_{j}\,h_{2})\circ\alpha_{j})_{H^{s,\varphi}(\mathbb{R}^{n-1})}.
$$
The space $H^{s,\varphi}(\Gamma)$ is Hilbert and separable. It does not depend (up to equivalence of norms) on the indicated choice of the local charts and partition of unity \cite[Theorem~2.3]{MikhailetsMurach10, MikhailetsMurach14}. The set $C^{\infty}(\Gamma)$ is dense in $H^{s,\varphi}(\Gamma)$.

The spaces $H^{s,\varphi}(\Omega)$ and $H^{s,\varphi}(\Gamma)$, with $s\in\mathbb{R}$ and $\varphi\in\mathcal{M}$, form the refined Sobolev scales over $\Omega$ and $\Gamma$ respectively. In the case of $\varphi(t)\equiv1$, these spaces become the Sobolev inner product spaces.

Note the following connection between these spaces \cite[Subsect. 3.2.1]{MikhailetsMurach10, MikhailetsMurach14}. If $s>1/2$, then the trace mapping $u\mapsto u\!\upharpoonright\!\Gamma$, with $u\in C^{\infty}(\Gamma)$, extends uniquely (by continuity) to a bounded surjective operator $R_{\Gamma}:H^{s,\varphi}(\Omega)\rightarrow H^{s-1/2,\varphi}(\Gamma)$. Thus, for each distribution $u\in H^{s,\varphi}(\Omega)$, its trace $R_{\Gamma}u$ on $\Gamma$ is well defined provided that $s>1/2$. But it is impossible to define this trace reasonably in the $s<1/2$ case. Hence, we cannot investigate the boundary-value problem \eqref{2f1}, \eqref{2f2} with $u\in H^{s,\varphi}(\Omega)$, where $s$ is an arbitrary real number.

Therefore, we use a certain modification of the space
$H^{s,\varphi}(\Omega)$ to be able to investigate this problem for arbitrary real $s$. This modification is denoted by $H^{s,\varphi,(2q)}(\Omega)$ and introduced in \cite{MikhailetsMurach08UMJ4} by analogy with Ya.~A.~Roitberg's \cite{Roitberg64, Roitberg65} construction applied to the Sobolev scale (see also monographs \cite[Chapt.~III, Sect.~6]{Berezansky68} and \cite[Sect.~2.1]{Roitberg96}).

Previously, let us define the space $H^{s,\varphi,(0)}(\Omega)$. We put $H^{s,\varphi,(0)}(\Omega):=H^{s,\varphi}(\Omega)$ in the case of $s\geq0$. But, in the case of $s<0$, we let $H^{s,\varphi,(0)}(\Omega)$ define the dual Hilbert space to $H^{-s,1/\varphi}(\Omega)$ with respect to the inner product in $L_{2}(\Omega)$. Namely, if $s<0$, then $H^{s,\varphi,(0)}(\Omega)$ is the completion of $C^{\infty}(\overline{\Omega})$ with respect to the Hilbert norm
$$
\|u\|_{H^{s,\varphi,(0)}(\Omega)}:=
\sup\left\{\,\frac{|(u,w)_{\Omega}|}
{\;\quad\quad\quad\|w\|_{H^{-s,1/\varphi}(\Omega)}}\,:
\,w\in H^{-s,1/\varphi}(\Omega),\,w\neq0\right\}.
$$
Note that the mapping $u\mapsto\mathcal{O}u$, with $u\in C^{\infty}(\overline{\Omega})$, extends uniquely (by continuity) to an isometric isomorphism of the space $H^{s,\varphi,(0)}(\Omega)$, with $s<0$, onto the subspace $\{w\in H^{s,\varphi}(\mathbb{R}^{n}):\mathrm{supp}\,w\subseteq \overline{\Omega}\}$ of $H^{s,\varphi}(\mathbb{R}^{n})$. Here, $\mathcal{O}u:=u$ on $\overline{\Omega}$, and $\mathcal{O}u:=0$ on $\mathbb{R}^{n}\setminus\overline{\Omega}$.

We can now define the Hilbert space $H^{s,\varphi,(2q)}(\Omega)$. Let
$E_{2q}:=\{k-1/2:k=1,\ldots,2q\}$. In the case of $s\in\mathbb{R}\setminus E_{2q}$, the space $H^{s,\varphi,(2q)}(\Omega)$ is defined to be the  completion of $C^{\infty}(\overline{\Omega})$ with respect to the Hilbert norm
$$
\|u\|_{H^{s,\varphi,(2q)}(\Omega)}:=
\biggl(\|u\|_{H^{s,\varphi,(0)}(\Omega)}^{2}+
\sum_{k=1}^{2q}\;\|(D_{\nu}^{k-1}u)\!\upharpoonright\!\Gamma\|
_{H^{s-k+1/2,\varphi}(\Gamma)}^{2}\biggr)^{1/2}.
$$
In the case of $s\in E_{2q}$, the space $H^{s,\varphi,(2q)}(\Omega)$ is defined with the help of interpolation between Hilbert spaces, namely,
$$
H^{s,\varphi,(2q)}(\Omega):=
\bigl[H^{s-\varepsilon,\varphi,(2q)}(\Omega),
H^{s+\varepsilon,\varphi,(2q)}(\Omega)\bigr]_{1/2},
\quad\mbox{with}\quad0<\varepsilon<1.
$$
We will give the definition of this interpolation in Section~\ref{sec5}.
Note \cite[Theorem 7.1]{MikhailetsMurach08UMJ4} that the right-hand side of the last equality does not depend (up to equivalence of norms) on the choice of $\varepsilon$. In the Sobolev case of $\varphi\equiv1$, the space $H^{s,\varphi,(2q)}(\Omega)$ was introduced by Ya.~A.~Roitberg \cite{Roitberg64, Roitberg65}.

For arbitrary $s\in\mathbb{R}$, the linear mapping
$T_{2q}:u\mapsto\bigl(u,u\!\upharpoonright\!\Gamma,\ldots,
(D_{\nu}^{2q-1}u)\!\upharpoonright\!\Gamma\bigr)$, with $u\in C^{\infty}(\overline{\Omega})$, extends uniquely (by continuity) to a bounded operator
$$
T_{2q}:\,H^{s,\varphi,(2q)}(\Omega)\rightarrow
H^{s,\varphi,(0)}(\Omega)\oplus
\bigoplus_{k=1}^{2q}\,H^{s-k+1/2,\varphi}(\Gamma)=:
\Pi_{s,\varphi,(2q)}(\Omega,\Gamma).
$$
If $s\notin E_{2q}$, then this operator is isometric and then its range consists of all vectors
$$
(u_{0},u_{1},\ldots,u_{2q})\in\Pi_{s,\varphi,(2q)}(\Omega,\Gamma)
$$
such that $u_{k}=R_{\Gamma}D_{\nu}^{k-1}u_{0}$ for all integers $k\in\{1,\ldots,2q\}$ satisfying $s>k-1/2$ (see \cite[Theorem 4.2]{MikhailetsMurach08UMJ4}).

The Hilbert spaces $H^{s,\varphi,(2q)}(\Omega)$ and $H^{s,\varphi,(0)}(\Omega)$ are separable. They form the refined scales
\begin{equation}\label{2f8}
\bigl\{H^{s,\varphi,(2q)}(\Omega):
\,s\in\mathbb{R},\varphi\in\mathcal{M}\bigr\}\quad\mbox{and}\quad
\bigl\{H^{s,\varphi,(0)}(\Omega):
\,s\in\mathbb{R},\varphi\in\mathcal{M}\bigr\}
\end{equation}
over $\Omega$. These scales are two-sided because the parameter $s$ runs through the whole real line. Note that $H^{s,\varphi,(2q)}(\Omega)=H^{s,\varphi}(\Omega)$ for $s>2q-1/2$ and that $H^{s,\varphi,(0)}(\Omega)=H^{s,\varphi}(\Omega)$ for $s>-1/2$, the equalities of Hilbert spaces being fulfilled up to equivalence of norms \cite[Theorems 3.9 and 4.12]{MikhailetsMurach10, MikhailetsMurach14}.

In the Sobolev case of $\varphi\equiv1$, we will omit $\varphi$ in the designations of the spaces introduced in the paper. It follows from \eqref{2f7} that
\begin{gather}\label{2f9}
H^{s+\varepsilon,(2q)}(\Omega)\hookrightarrow H^{s,\varphi,(2q)}(\Omega)\hookrightarrow H^{s-\varepsilon,(2q)}(\Omega),\\\label{2f10}
H^{s+\varepsilon,(0)}(\Omega)\hookrightarrow H^{s,\varphi,(0)}(\Omega)\hookrightarrow H^{s-\varepsilon,(0)}(\Omega),\\
H^{s+\varepsilon}(\Gamma)\hookrightarrow H^{s,\varphi}(\Gamma)\hookrightarrow H^{s-\varepsilon}(\Gamma)\label{2f11}
\end{gather}
for every $\varepsilon>0$. These embeddings are compact and dense.

\section{The main results}\label{sec4}

Here, we will formulate our results concerning the properties of the elliptic boundary-value problem \eqref{2f1}, \eqref{2f2} on the two-sided refined scales \eqref{2f8} and $\{H^{s,\varphi}(\Gamma):s\in\mathbb{R},\varphi\in\mathcal{M}\}$.

With the problem \eqref{2f1}, \eqref{2f2} and its formally adjoint problem \eqref{2f4}, \eqref{2f5}, \eqref{2f6}, we associate the linear spaces $N$ and $N^{+}$ of infinitely smooth solutions to these problems in the homogeneous case. Namely, let $N$ consist of all solutions
$(u,v_{1},...,v_{\varkappa})\in C^{\infty}(\overline{\Omega})\times(C^{\infty}(\Gamma))^{\varkappa}$
to the problem \eqref{2f1}, \eqref{2f2} in the case where $f=0$ in $\Omega$ and all $g_{j}=0$ on~$\Gamma$. Besides, let $N^{+}$ consist of all solutions $(w,h_{1},...,h_{q+\varkappa})\in C^{\infty}(\overline{\Omega})\times(C^{\infty}(\Gamma))^{q+\varkappa}$
to the problem \eqref{2f4}, \eqref{2f5}, \eqref{2f6} in the case where $\omega=0$ in $\Omega$ and all $\chi_{j}=0$ and all $\chi_{2q+j}=0$ on $\Gamma$.
Since these problems are elliptic in $\Omega$, the spaces $N$ and $N^{+}$ are finite-dimensional \cite[Lemma~3.4.2]{KozlovMazyaRossmann97}.

\begin{theorem}\label{2th1}
For arbitrary $s\in\mathbb{R}$ and $\varphi\in\mathcal{M}$, the mapping \eqref{2f3} extends uniquely (by continuity) to a bounded operator
\begin{equation}\label{2f12}
\begin{gathered}
\Lambda:\,\mathcal{D}^{s,\varphi,(2q)}(\Omega,\Gamma):=H^{s,\varphi,(2q)}(\Omega)\oplus
\bigoplus_{k=1}^{\varkappa}H^{s+r_{k}-1/2,\varphi}(\Gamma)\rightarrow\\
\rightarrow H^{s-2q,\varphi,(0)}(\Omega)\oplus
\bigoplus_{j=1}^{q+\varkappa}H^{s-m_{j}-1/2,\varphi}(\Gamma)
=:\mathcal{E}^{s-2q,\varphi,(0)}(\Omega,\Gamma).
\end{gathered}
\end{equation}
This operator is Fredholm. Its kernel coincides with $N$, and its range consists of all vectors
\begin{equation}\label{2f13}
(f,g):=(f,g_{1},\ldots,g_{q+\varkappa})
\in\mathcal{E}^{s-2q,\varphi,(0)}(\Omega,\Gamma)
\end{equation}
such that
\begin{equation}\label{2f14}
(f,w)_{\Omega}+\sum_{j=1}^{q+\varkappa}(g_{j},h_{j})_{\Gamma}=0\quad
\mbox{for each}\quad(w,h_{1},\ldots,h_{q+\varkappa})\in N^{+}.
\end{equation}
The index of the operator \eqref {2f12} is equal to $\dim N-\dim N^{+}$ and does not depend on $s$ and~$\varphi$.
\end{theorem}

Here and below, we let $(\cdot,\cdot)_{\Omega}$ and $(\cdot,\cdot)_{\Gamma}$ denote the extension by continuity of the inner products in the Hilbert spaces $L_{2}(\Omega)$ and $L_{2}(\Gamma)$ respectively.

In connection with this theorem, we recall that a linear bounded operator $T:E_{1}\rightarrow E_{2}$ between some Banach spaces $E_{1}$ and $E_{2}$ is said to be Fredholm if its kernel $\ker T$ and co-kernel $E_{2}/T(E_{1})$ are finite-dimensional. If this operator is Fredholm, then its range $T(E_{1})$ is closed in $E_{2}$, and its index $\mathrm{ind}\,T:=\dim\ker T-\dim(E_{2}/T(E_{1}))$ is finite.

It is useful to note that, in the case of $s>2q-1/2$, the bounded and Fredholm operator \eqref{2f12} acts between the spaces
$$
\Lambda:\,H^{s,\varphi}(\Omega)\oplus
\bigoplus_{k=1}^{\varkappa}H^{s+r_{k}-1/2,\varphi}(\Gamma)\rightarrow
H^{s-2q,\varphi}(\Omega)\oplus
\bigoplus_{j=1}^{q+\varkappa}H^{s-m_{j}-1/2,\varphi}(\Gamma),
$$
which consist of distributions on $\Omega$ or $\Gamma$.

In the case where $N=\{0\}$ and $N^{+}=\{0\}$, the operator \eqref{2f12} is an isomorphism of the space $\mathcal{D}^{s,\varphi,(2q)}(\Omega,\Gamma)$ onto the space $\mathcal{E}^{s-2q,\varphi,(0)}(\Omega,\Gamma)$. Generally, this operator induces an isomorphism between some of their (closed) subspaces, which have finite co-dimension. In this connection, we consider the following decompositions of these spaces in direct sums of subspaces:
\begin{equation}\label{2f15}
\begin{gathered}
\mathcal{D}^{s,\varphi,(2q)}(\Omega,\Gamma)=
N\dotplus\biggl\{(u,v_{1},\ldots,v_{\varkappa})\in
\mathcal{D}^{s,\varphi,(2q)}(\Omega,\Gamma):\\
(u_{0},u^{(0)})_\Omega+\sum_{k=1}^{\varkappa}(v_{k},v_{k}^{(0)})_{\Gamma}=0\;\,
\mbox{for all}\;\,(u^{(0)},v_{1}^{(0)},\ldots,v_{\varkappa}^{(0)})\in N\biggr\},
\end{gathered}
\end{equation}
with $u_{0}$ being the initial component of the vector $T_{2q}u=(u_{0},u_{1},\ldots,u_{2q})$, and
\begin{equation}\label{2f16}
\mathcal{E}^{s-2q,\varphi,(0)}(\Omega,\Gamma)=
N^{+}\dotplus\Lambda\bigl(\mathcal{D}^{s,\varphi,(2q)}(\Omega,\Gamma)\bigr).
\end{equation}
Each decomposition is well defined because the subspaces have the intersection $\{0\}$ and because the finite dimension of the first subspace is equal to the co-dimension of the second. Let $P$ and $P^{+}$ denote respectively the projectors of the spaces
$\mathcal{D}^{s,\varphi,(2q)}(\Omega,\Gamma)$ and $\mathcal{E}^{s-2q,\varphi,(0)}(\Omega,\Gamma)$ onto the second term in the sums \eqref{2f15} and \eqref{2f16} parallel to the first term. These projectors do not depend, as mappings, on $s$ and $\varphi$.

\begin{theorem}\label{2th2}
For arbitrary $s\in\mathbb{R}$ and $\varphi\in\mathcal {M}$, the restriction of the mapping \eqref{2f12} to the subspace  $P(\mathcal{D}^{s,\varphi,(2q)}(\Omega,\Gamma))$ is an isomorphism
\begin{equation}\label{2f17}
\Lambda:\,P\bigl(\mathcal{D}^{s,\varphi,(2q)}(\Omega,\Gamma)\bigr)
\leftrightarrow P^{+}\bigl(\mathcal{E}^{s-2q,\varphi,(0)}(\Omega,\Gamma)\bigr).
\end{equation}
\end{theorem}

We will discuss properties of the generalized solutions to the elliptic  boundary-value problem \eqref{2f1}, \eqref{2f2}. Beforehand, we give the definition of such solutions.

Let $\mathcal{D}^{-\infty,(2q)}(\Omega,\Gamma)$ denote the union of all the spaces $\mathcal{D}^{s,\varphi,(2q)}(\Omega)$ with $s\in\mathbb{R}$ and $\varphi\in\mathcal{M}$. Similarly, let $\mathcal{E}^{-\infty,(0)}(\Omega,\Gamma)$ denote the union of all the spaces $\mathcal{E}^{s,\varphi,(0)}(\Omega)$ with $s\in\mathbb{R}$ and $\varphi\in\mathcal{M}$. The linear spaces  $\mathcal{D}^{-\infty,(2q)}(\Omega,\Gamma)$ and $\mathcal{E}^{-\infty,(0)}(\Omega,\Gamma)$ are well defined in view of the embeddings \eqref{2f9}--\eqref{2f11} and are endowed with the topology of inductive limit.

Consider an arbitrary vector
\begin{equation}\label{2f18}
(u,v):=(u,v_{1},\ldots,v_{\varkappa})\in
\mathcal{D}^{-\infty,(2q)}(\Omega,\Gamma).
\end{equation}
Since $(u,v)\in\mathcal{D}^{s,\varphi,(2q)}(\Omega)$ for some $s\in\mathbb{R}$ and $\varphi\in\mathcal{M}$, its image \eqref{2f13}
is defined by the formula $\Lambda(u,v)=(f,g)$ by means of the bounded operator \eqref{2f12}. The image $(f,g)$ does not depend on $s\in\mathbb{R}$ and $\varphi\in\mathcal{M}$. The vector \eqref{2f18} is said to be a generalized solution (in the sense of Ya.~A.~Roitberg) of the boundary-value problem \eqref{2f1}, \eqref{2f2} with the right-hand sides \eqref{2f13}. Such a solution satisfies the following a priori estimate.

\begin{theorem}\label{2th3}
Let $s\in\mathbb{R}$, $\varphi\in\mathcal{M}$, and a real number $\sigma>0$. Then there exists a number $c>0$ such that
\begin{equation}\label{2f19}
\begin{gathered}
\|(u,v)\|_{\mathcal{D}^{s,\varphi,(2q)}(\Omega,\Gamma)}\leq c\,\bigl(\,\|\Lambda(u,v)\|_{\mathcal{E}^{s-2q,\varphi,(0)}(\Omega,\Gamma)}+
\|(u,v)\|_{\mathcal{D}^{s-\sigma,\varphi,(2q)}(\Omega,\Gamma)}\bigl)
\end{gathered}
\end{equation}
for an arbitrary vector $(u,v)\in\mathcal{D}^{s,\varphi,(2q)}(\Omega,\Gamma)$. Here, the number $c=c(s,\varphi)$ does not depend on~$(u,v)$.
\end{theorem}

Let us now focus our attention on the regularity properties of the generalized solutions to the elliptic problem under investigation. Let $V$ be an arbitrary open subset of $\mathbb{R}^{n}$. We put
$\Omega_{0}:=\Omega\cap V\neq\varnothing$ and $\Gamma_{0}:=\Gamma\cap V$, the $\Gamma_{0}=\varnothing$ case being possible. Introduce a local analogs of the spaces $\mathcal{D}^{s,\varphi,(2q)}(\Omega)$ and $\mathcal{E}^{s,\varphi,(0)}(\Omega)$, with $s\in\mathbb{R}$ and $\varphi\in\mathcal{M}$. We let $\mathcal{D}^{s,\varphi,(2q)}_{\mathrm{loc}}(\Omega_{0},\Gamma_{0})$ denote the linear space of all vectors \eqref{2f18} such that
$\chi(u,v)\in\mathcal{D}^{s,\varphi,(2q)}(\Omega)$ for an arbitrary function $\chi\in C^{\infty}(\overline{\Omega})$ with $\mathrm{supp}\,\chi\subset\Omega_0\cup\Gamma_{0}$. Similarly, we let
$\mathcal{E}^{s,\varphi,(0)}_{\mathrm{loc}}(\Omega_{0},\Gamma_{0})$
denote the linear space of all vectors $(f,g)\in\mathcal{E}^{-\infty,(0)}(\Omega,\Gamma)$
such that
$\chi(f,g)\in\mathcal{E}^{s,\varphi,(0)}(\Omega)$ for an arbitrary function $\chi$ indicated just above. Of course, the multiplication of the vectors by the scalar function $\chi$ is understood to be component-vise,  the products $\chi u$ and $\chi f$ are well defined by closure, and the product $\chi h$, where $h$ is a distribution on $\Gamma$, means $(\chi\!\upharpoonright\Gamma\!)h$. We introduce topologies in the spaces  $\mathcal{D}^{s,\varphi,(2q)}_{\mathrm{loc}}(\Omega_{0},\Gamma_{0})$ and
$\mathcal{E}^{s,\varphi,(0)}_{\mathrm{loc}}(\Omega_{0},\Gamma_{0})$
by means of the seminorms $(u,v)\mapsto\|\chi(u,v)\|_{\mathcal{D}^{s,\varphi,(2q)}(\Omega)}$ and
$(f,g)\mapsto\|\chi(f,g)\|_{\mathcal{E}^{s,\varphi,(0)}(\Omega)}$ respectively, were $\chi$ is an arbitrary function from the definitions of these spaces.

\begin{theorem}\label{2th4}
Let a vector $(u,v)\in\mathcal{D}^{-\infty,(2q)}(\Omega,\Gamma)$ be a generalized solution to the elliptic problem \eqref{2f1}, \eqref{2f2} whose right-hand sides satisfy the condition  $(f,g)\in
\mathcal{E}^{s-2q,\varphi,(0)}_{\mathrm{loc}}(\Omega_{0},\Gamma_{0})$ for some parameters $s\in\mathbb{R}$ and $\varphi\in \mathcal{M}$. Then $(u,v)\in\mathcal{D}^{s,\varphi,(2q)}_{\mathrm{loc}}(\Omega_{0},\Gamma_{0})$.
\end{theorem}

We can see that the solution inherits the (supplementary) regularity $\varphi$ of the right-hand sides of the elliptic problem.

In the case where $\Omega_{0}=\Omega$ and $\Gamma_{0}=\Gamma$, we have  the equalities $\mathcal{D}^{s,\varphi,(2q)}_{\mathrm{loc}}(\Omega_{0},\Gamma_{0})=
\mathcal{D}^{s,\varphi,(2q)}(\Omega,\Gamma)$ and $\mathcal{E}^{s-2q,\varphi,(0)}_{\mathrm{loc}}(\Omega_{0},\Gamma_{0})=
\mathcal{E}^{s-2q,\varphi,(0)}(\Omega,\Gamma)$. Therefore, Theorem~\ref{2f4} says about the global regularity in this case, i.e., the regularity on the whole closed domain~$\overline{\Omega}$.

If $\Gamma_{0}=\varnothing$, then this theorem states that the local regularity of the component $u$ of the generalized solution increases in neighbourhoods of internal points of the domain~$\Omega$.

In the Sobolev case of $\varphi(t)\equiv1$, Theorem 1--4 are proved in monographs \cite[Sect. 3.2 and 3.4]{KozlovMazyaRossmann97} (for an arbitrary integer $s$) and \cite[Sect. 2.4]{Roitberg99} (for arbitrary real $s$ and general elliptic systems). For arbitrary $\varphi\in\mathcal{M}$ and real $s>2q$, Theorems 1 and 2 with analogs of Theorems 3 and 4 are proved in \cite{Chepurukhina14Coll2}.

\section{Auxiliary results}\label{sec5}

The refined scales of spaces introduced in Section~\ref{sec3} possess an important interpolation property, which will play a key role in the proof of Theorem~\ref{2th1}. Namely, each of these spaces can be obtained by the interpolation, with a function parameter, of an appropriate couple of Hilbert spaces corresponding to the Sobolev case of $\varphi(t)\equiv1$.
Therefore, we will recall the definition of this interpolation in the case of arbitrary Hilbert spaces and discuss some of its properties that are necessary for us. Among them, we will formulate the mentioned interpolation property for Hilbert spaces appearing in \eqref{2f12}.

The method of interpolation with a function parameter of Hilbert spaces was introduced by C.~Foia\c{s} and J.-L.~Lions in \cite[p.~278]{FoiasLions61}. We follow monographs \cite[Sect.~1.1]{MikhailetsMurach10, MikhailetsMurach14}, which systematically expound this method. It is sufficient for our purposes to restrict ourselves to separable Hilbert spaces.

Let $X:=[X_{0},X_{1}]$ be a given ordered couple of separable complex Hilbert spaces $X_{0}$ and $X_{1}$ such that the continuous and dense embedding $X_{1}\hookrightarrow X_{0}$ holds. This couple is called admissible. For $X$, there exists a self-adjoint positive-definite operator $J$ on $X_{0}$ which has the domain $X_{1}$ and satisfies the condition $\|Jw\|_{X_{0}}=\|w\|_{X_{1}}$ for all $w\in X_{1}$. This operator is
uniquely determined by the couple $X$ and is called a generating operator for~$X$. It sets an isometric isomorphism $J:X_{1}\leftrightarrow X_{0}$.

Let $\mathcal{B}$ denote the set of all Borel measurable functions
$\psi:(0,\infty)\rightarrow(0,\infty)$ such that $\psi$ is bounded on each compact interval $[a,b]$, with $0<a<b<\infty$, and that $1/\psi$ is bounded on every set $[r,\infty)$, with $r>0$.

For arbitrary $\psi\in\mathcal{B}$, we consider the (generally, unbounded) operator $\psi(J)$, which is defined on $X_{0}$ as the Borel function $\psi$ of $J$. Let $[X_{0},X_{1}]_{\psi}$ or, simply, $X_{\psi}$ denote the domain of the operator $\psi(J)$ endowed with the inner product
$(u_{1},u_{2})_{X_{\psi}}:=(\psi(J)u_{1},\psi(J)u_{2})_{X_{0}}$ and the
corresponding norm $\|u\|_{X_{\psi}}=\|\psi(J)u\|_{X_{0}}$. The space $X_{\psi}$ is Hilbert and separable.

A function $\psi\in\mathcal{B}$ is called an interpolation parameter if the
following condition is fulfilled for all admissible couples $X=[X_{0},X_{1}]$ and
$Y=[Y_{0},Y_{1}]$ of Hilbert spaces and for an arbitrary linear mapping $T$ given on
$X_{0}$: if the restriction of $T$ to $X_{j}$ is a bounded operator
$T:X_{j}\rightarrow Y_{j}$ for each $j\in\{0,1\}$, then the restriction of $T$ to
$X_{\psi}$ is also a bounded operator $T:X_{\psi}\rightarrow Y_{\psi}$. In this case
we say that $X_{\psi}$ is obtained by the interpolation with the function parameter $\psi$ of the couple $X$.

The function $\psi\in\mathcal{B}$ is an interpolation parameter if and only if $\psi$ is pseudoconcave on a neighbourhood of infinity, i.e.
$\psi(t)\asymp\psi_{1}(t)$ with $t\gg1$ for a certain positive concave
function~$\psi_{1}(t)$. (As usual, the designation $\psi\asymp\psi_{1}$ means that both the functions $\psi/\psi_{1}$ and $\psi_{1}/\psi$ are bounded on the indicated set). This fundamental fact follows from J.~Peetre's \cite{Peetre68} description of all interpolation functions of positive order. Specifically, the power function $\psi(t)\equiv t^{s}$ is an interpolation parameter if and only if $0\leq s\leq1$. In this case,  the interpolation space $X_{\psi}$ is denoted by $X_{s}$.

We need the following interpolation properties of the refined scales \cite[Theorems 1.14, 2.2, 3.10, and 4.22]{MikhailetsMurach10, MikhailetsMurach14}.

\begin{proposition}\label{2pr1}
Let a function $\varphi\in\mathcal{M}$ and positive real numbers $\varepsilon,\delta$ be given. Define a function $\psi\in\mathcal{B}$ by the formulas $\psi(t):=
t^{\varepsilon/(\varepsilon+\delta)}\varphi(t^{1/(\varepsilon+\delta)})$ for $t>1$, and $\psi(t):=\varphi(1)$ if $0<t<1$. Then $\psi$ is an interpolation parameter, and, for arbitrary $\sigma\in\mathbb{R}$, we have  \begin{gather*}
\bigl[H^{\sigma-\varepsilon,(2q)}(\Omega),
H^{\sigma+\delta,(2q)}(\Omega)\bigr]_{\psi}
=H^{\sigma,\varphi,(2q)}(\Omega),\\
\bigl[H^{\sigma-\varepsilon,(0)}(\Omega),
H^{\sigma+\delta,(0)}(\Omega)\bigr]_{\psi}
=H^{\sigma,\varphi,(0)}(\Omega),\\
\bigl[H^{\sigma-\varepsilon}(\Gamma),
H^{\sigma+\delta}(\Gamma)\bigr]_{\psi}
=H^{\sigma,\varphi}(\Gamma)
\end{gather*}
up to equivalence of norms.
\end{proposition}

We will also use the following two general properties of the interpolation \cite[Theorems 1.7 and 1.5]{MikhailetsMurach10, MikhailetsMurach14}.

\begin{proposition}\label{2pr2}
Let $X=[X_{0},X_{1}]$ and $Y=[Y_{0},Y_{1}]$ be admissible couples of Hilbert spaces, and let a linear mapping $T$ be given on $X_{0}$. Suppose that we have the bounded and Fredholm operators $T:X_{j}\rightarrow Y_{j}$, with $j=0,\,1$, that possess the common kernel and the common index. Then, for an arbitrary interpolation parameter $\psi\in\mathcal{B}$, the bounded operator $T:X_{\psi}\rightarrow Y_{\psi}$ is Fredholm, has the same kernel and the same index, and, moreover, its range $T(X_{\psi})=Y_{\psi}\cap T(X_{0})$.
\end{proposition}

\begin{proposition}\label{2pr3}
Let $\bigl[X_{0}^{(j)},X_{1}^{(j)}\bigr]$, with $j=1,\ldots,r$, be a finite collection of admissible couples of Hilbert spaces. Then, for every function $\psi\in\mathcal{B}$, we have
$$
\biggl[\,\bigoplus_{j=1}^{r}X_{0}^{(j)},
\,\bigoplus_{j=1}^{r}X_{1}^{(j)}\biggr]_{\psi}=\,
\bigoplus_{j=1}^{r}\bigl[X_{0}^{(j)},X_{1}^{(j)}\bigr]_{\psi}
$$
with equality of norms.
\end{proposition}

\section{Proofs of the main results}\label{sec6}

Let us prove Theorems \ref{2th1}--\ref{2th4}.

\begin{proof}[Proof of Theorem $\ref{2th1}$]
In the Sobolev case of $\varphi(t)\equiv1$, this theorem is proved in \cite[Theorem 3.4.1]{KozlovMazyaRossmann97} (for integer-valued $s$) and \cite[Theorem 2.4.1]{Roitberg99} (for real $s$ and elliptic systems). We will deduce Theorem~\ref{2th1} from the Sobolev case with the help of the interpolation with a function parameter.

Let $s\in\mathbb{R}$ and $\varphi\in\mathcal{M}$. We choose positive numbers $\varepsilon$ and $\delta$ such that $s-\varepsilon$ and $s+\delta$ are integers. According to \cite[Theorem 3.4.1]{KozlovMazyaRossmann97} the mapping \eqref{2f3} extends by continuity to the bounded and Fredholm operators
\begin{equation}\label{2f20}
\begin{gathered}
\Lambda:\mathcal{D}^{s-\varepsilon,(2q)}(\Omega,\Gamma)\rightarrow
\mathcal{E}^{s-\varepsilon-2q,(0)}(\Omega,\Gamma),\quad
\Lambda:\mathcal{D}^{s+\delta,(2q)}(\Omega,\Gamma)\rightarrow
\mathcal{E}^{s+\delta-2q,(0)}(\Omega,\Gamma).
\end{gathered}
\end{equation}
(Recall that we omit the index $\varphi$ in the designations of the spaces if $\varphi\equiv1$.) These operators have the common kernel $N$ and the same index equaled to $\dim N-\dim N^{+}$. The range of the first operator satisfies the equality
\begin{equation}\label{2f21}
\Lambda(\mathcal{D}^{s-\varepsilon,(2q)}(\Omega,\Gamma))=
\bigl\{(f,g)\in\mathcal{E}^{s-\varepsilon-2q,(0)}(\Omega,\Gamma):\,
\eqref{2f14}\;\mbox{is true}\bigr\},
\end{equation}
and analogous formula holds for the second operator.

Let us apply the interpolation with the function parameter $\psi$ from Proposition~\ref{2pr1} to~\eqref{2f20}. According to Proposition \ref{2pr2}, we get the bounded and Fredholm operator
\begin{equation}\label{2f22}
\Lambda:\,\bigl[\mathcal{D}^{s-\varepsilon,(2q)}(\Omega,\Gamma),
\mathcal{D}^{s+\delta,(2q)}(\Omega,\Gamma)\bigr]_{\psi}\rightarrow
\bigl[\mathcal{E}^{s-\varepsilon-2q,(0)}(\Omega,\Gamma),
\mathcal{E}^{s+\delta-2q,(0)}(\Omega,\Gamma)\bigr]_{\psi}.
\end{equation}
This operator is an extension of the mapping \eqref{2f3} by continuity.

Let us describe the interpolation spaces appearing in \eqref{2f22}. Applying Propositions \ref{2pr3} and \ref{2pr1} successively, we obtain
\begin{gather*}
\bigl[\mathcal{D}^{s-\varepsilon,(2q)}(\Omega,\Gamma),
\mathcal{D}^{s+\delta,(2q)}(\Omega,\Gamma)\bigr]_{\psi}=\\
=\bigl[H^{s-\varepsilon,(2q)}(\Omega),H^{s+\delta,(2q)}(\Omega)\bigr]_{\psi}
\oplus\bigoplus_{k=1}^{\varkappa}\,
\bigl[H^{s+r_{k}-1/2-\varepsilon}(\Gamma),
H^{s+r_{k}-1/2+\delta}(\Gamma)\bigr]_{\psi}=\\
=H^{s,\varphi,(2q)}(\Omega)\oplus
\bigoplus_{k=1}^{\varkappa}H^{s+r_{k}-1/2,\varphi}(\Gamma)=
\mathcal{D}^{s,\varphi,(2q)}(\Omega,\Gamma).
\end{gather*}
Analogously,
$$
[\mathcal{E}^{s-\varepsilon-2q,(0)}(\Omega,\Gamma),
\mathcal{E}^{s+\delta-2q,(0)}(\Omega,\Gamma)]_{\psi}=
\mathcal{E}^{s-2q,\varphi,(0)}(\Omega,\Gamma).
$$
These equalities of Hilbert spaces are fulfilled up to equivalence of norms.

Thus, the bounded and Fredholm operator \eqref{2f22} is the operator \eqref{2f12} from Theorem~\ref{2th1}. By virtue of Proposition~\ref{2pr2}, the kernel and index of the operator \eqref{2f12} coincide respectively with the common kernel $N$ and the common index $\dim N-\dim N^+$ of the operators \eqref{2f20}. Moreover, the range of \eqref{2f12} is equal to
$$
\mathcal{E}^{s-2q,\varphi,(0)}(\Omega,\Gamma)\cap
\Lambda\bigl(\mathcal{D}^{s-\varepsilon,(2q)}(\Omega,\Gamma)\bigr)=
\bigl\{(f,g)\in\mathcal{E}^{s-2q,\varphi,(0)}(\Omega,\Gamma):\,\eqref{2f14}
\;\mbox{is true}\bigr\}
$$
in view of \eqref{2f21}.
\end{proof}

\begin{proof}[Proof of Theorem $\ref{2th2}$]
According to Theorem~\ref{2th1}, the restriction of the operator  \eqref{2f12} to $P(\mathcal{D}^{s,\varphi,(2q)}(\Omega,\Gamma))$ is a continuous and one-to-one mapping from the whole subspace  $P(\mathcal{D}^{s,\varphi,(2q)}(\Omega,\Gamma))$ to the whole subspace
$P^{+}(\mathcal{E}^{s-2q,\varphi,(0)}(\Omega,\Gamma))$. Then, by the Banach theorem on inverse operator, this linear mapping is the isomorphism~\eqref{2f17}.
\end{proof}

\begin{proof}[Proof of Theorem $\ref{2th3}$]
Choose vector $(u,v)\in \mathcal{D}^{s,\varphi,(2q)}(\Omega,\Gamma)$ arbitrarily. According to Theorem~\ref{2th2}, we have the inequalities
\begin{equation}\label{2f23}
\begin{gathered}
\|(u,v)\|_{\mathcal{D}^{s,\varphi,(2q)}(\Omega,\Gamma)}\leq
\|P(u,v)\|_{\mathcal{D}^{s,\varphi,(2q)}(\Omega,\Gamma)}+
\|(u,v)-P(u,v)\|_{\mathcal{D}^{s,\varphi,(2q)}(\Omega,\Gamma)}\leq\\
\leq
c_{1}\|\Lambda P(u,v)\|_{\mathcal{E}^{s-2q,\varphi,(0)}(\Omega,\Gamma)}+
c_{2}\|(u,v)-P(u,v)\|_{\mathcal{D}^{s-\sigma,\varphi,(2q)}(\Omega,\Gamma)}.
\end{gathered}
\end{equation}
Here, $c_1$ is the norm of the inverse operator to the isomorphism \eqref{2f17}, whereas $c_2$ is a certain positive number not depending on  $(u,v)$. This number exists because the vector $(u,v)-P(u,v)$ belongs to the finite-dimensional space $N$, where all norms are equivalent, specifically, the norms in $\mathcal{D}^{s,\varphi,(2q)}(\Omega,\Gamma)$ and $\mathcal{D}^{s-\sigma,\varphi,(2q)}(\Omega,\Gamma)$. Besides,
\begin{equation}\label{2f24}
\|P(u,v)\|_{\mathcal{D}^{s-\sigma,\varphi,(2q)}(\Omega,\Gamma)}\leq
c_3\|(u,v)\|_{\mathcal{D}^{s-\sigma,\varphi,(2q)}(\Omega,\Gamma)},
\end{equation}
where $c_3$ is the norm of the projector $P$ acting on the space $\mathcal{D}^{s-\sigma,\varphi,(2q)}(\Omega,\Gamma)$. Now, the formulas \eqref{2f23} and \eqref{2f24} directly imply the required estimate \eqref{2f19}.
\end{proof}

\begin{proof}[Proof of Theorem $\ref{2th4}$]
First, let us prove this theorem in the case of global regularity, where
$\Omega_{0}=\Omega$ and $\Gamma_{0}=\Gamma$. According to Theorem~\ref{2th1}, the vector $(f,g):=\Lambda(u,v)$ satisfies \eqref{2f14}. Therefore, by virtue of the same theorem and the condition $(f,g)\in\mathcal{E}^{s-2q,\varphi,(0)}(\Omega,\Gamma)$, we get the inclusion $(f,g)\in\Lambda(\mathcal{D}^{s,\varphi,(2q)}(\Omega,\Gamma))$.
Hence, together with $\Lambda(u,v)=(f,g)$, we have the equality  $\Lambda(u',v')=(f,g)$ for a certain vector   $(u',v')\in\mathcal{D}^{s,\varphi,(2q)}(\Omega,\Gamma)$. Therefore, $\Lambda(u-\nobreak u',v-v')=0$, that implies
$$
(u-u',v-v')\in N\subset C^{\infty}(\overline{\Omega})\times
\bigl(C^{\infty}(\Gamma)\bigr)^{\varkappa}
$$
in view of Theorem~\ref{2th1}. Thus,
$$
(u,v)=(u',v')+(u-u',v-v')\in\mathcal{D}^{s,\varphi,(2q)}(\Omega,\Gamma).
$$
Theorem \ref{2th4} is proved in the case of global regularity.

We now pass to the general case of local regularity and deduce the theorem in this case from the specific case of global regularity just considered.
We will previously prove that the following implication holds true for each $p\geq1$:
\begin{equation}\label{2f25}
(u,v)\in
\mathcal{D}^{s-p,\varphi,(2q)}_{\mathrm{loc}}(\Omega_{0},\Gamma_{0})
\;\Rightarrow\;
(u,v)\in
\mathcal{D}^{s-p+1,\varphi,(2q)}_{\mathrm{loc}}(\Omega_{0},\Gamma_{0})
\end{equation}
under the conditions of the theorem.

Suppose that the premise of this implication is valid for a certain number $p\geq1$. We arbitrarily choose a function $\chi\in C^{\infty}(\overline{\Omega})$ such that $\mathrm{supp}\,\chi\subset\Omega_{0}\cup\Gamma_{0}$. We also choose a certain function $\eta\in C^{\infty}(\overline{\Omega})$ that satisfies the conditions $\mathrm{supp}\,\eta\subset\Omega_{0}\cup\Gamma_{0}$ and $\eta=1$ on a neighbourhood of $\mathrm{supp}\,\chi$.

Interchanging the operator of the multiplication by $\chi$ with each of the differential operators $A$, $B_{j}$, and $C_{j,k}$, we can write down  the following equalities:
\begin{gather}\label{2f26}
A(\chi u)=A(\chi\eta u)=\chi A(\eta u)+A'(\eta u)=\chi Au+A'(\eta u),\\
B_{j}(\chi u)=B_{j}(\chi\eta u)=\chi B_{j}(\eta u)+B'_{j}(\eta u)=
\chi B_{j}u+B'_{j}(\eta u),\label{2f27}\\
\begin{gathered}\label{2f28}
C_{j,k}(\chi v_k)=C_{j,k}(\chi\eta v_k)=\chi C_{j,k}(\eta v_k)+
C'_{j,k}(\eta v_k)=
\chi C_{j,k}v_k+C'_{j,k}(\eta v_k).
\end{gathered}
\end{gather}
Here, $A'$ is a certain differential operator on  $\overline{\Omega}$, $B'_{j}$ is a certain boundary differential operator on $\Gamma$, and $C'_{j,k}$ is a certain tangent differential operator on $\Gamma$. These operators are linear, all of their coefficients are infinitely smooth on the indicated sets, and the orders of these operators satisfy the conditions
\begin{equation}\label{2f29}
\mathrm{ord}\,A'\leq 2q-1,\quad\mathrm{ord}\,B'_{j}\leq m_{j}-1,\quad
\mathrm{ord}\,C'_{j,k}\leq m_{j}+r_{k}-1.
\end{equation}

We put
$$
\Lambda'(u,v):=\biggl(A'u,\,B'_{1}u+\sum_{k=1}^{\varkappa}C'_{1,k}v_{k},...,
B'_{q+\varkappa}u+\sum_{k=1}^{\varkappa}C'_{q+\varkappa,k}v_{k}\biggr).
$$
It follows from \eqref{2f26}--\eqref{2f28} that
\begin{equation}\label{2f30}
\Lambda(\chi u,\chi v)=\chi\Lambda(u,v)+\Lambda'(\eta u,\eta v)=
(\chi f,\chi g)+\Lambda'(\eta u,\eta v).
\end{equation}
Here, by the condition of the theorem,
\begin{equation}\label{2f31}
(\chi f,\chi g)\in\mathcal{E}^{s-2q,\varphi,(0)}(\Omega,\Gamma)\subseteq
\mathcal{E}^{s-p+1-2q,\varphi,(0)}(\Omega,\Gamma).
\end{equation}
According to the premise of the implication,  we have $(\eta u,\eta v)\in\mathcal{D}^{s-p,\varphi,(2q)}(\Omega,\Gamma)$. It follows from the latter inclusion and the inequalities \eqref{2f29} that
\begin{equation}\label{2f32}
\Lambda'(\eta u,\eta v)\in \mathcal{E}^{s-p-(2q-1),\varphi,(0)}(\Omega,\Gamma)=
\mathcal{E}^{s-p+1-2q,\varphi,(0)}(\Omega,\Gamma)
\end{equation}
in view of \cite[Theorem 4.13 and Lemma 2.5]{MikhailetsMurach10, MikhailetsMurach14}.

The formulas \eqref{2f30}--\eqref{2f32} gives the inclusion
\begin{equation}\label{2f33}
\Lambda(\chi u,\chi v)\in \mathcal{E}^{s-p+1-2q,\varphi,(0)}(\Omega,\Gamma).
\end{equation}
By virtue of Theorem \ref{2th4} in the global case proved above, we conclude that \eqref{2f33} implies the inclusion $(\chi u,\chi v)\in\mathcal{D}^{s-p+1,\varphi,(2q)}(\Omega,\Gamma)$. In view of arbitrariness of the indicated choice of $\chi$, this inclusion means that the inference of the implication \eqref{2f25} is true. Thus, we have proved this implication for each $p\geq1$.

Now, using the implication \eqref{2f25}, we can prove Theorem \ref{2th4} in the general case. By the condition of this theorem, $(u,v)\in\mathcal{D}^{-\infty,(2q)}(\Omega,\Gamma)$. It follows from this inclusion, in view of \eqref{2f9} and \eqref{2f11}, that
$$
(u,v)\in\mathcal{D}^{s-\lambda,\varphi,(2q)}(\Omega,\Gamma)\subseteq
\mathcal{D}^{s-\lambda,\varphi,(2q)}_{\mathrm{loc}}(\Omega_{0},\Gamma_{0})
$$
for a certain integer $\lambda\geq1$. Using the implication \eqref{2f25} successively for values $p=\lambda$, $p=\lambda-1$, ..., and $p=1$, we conclude that
\begin{gather*}
(u,v)\in
\mathcal{D}^{s-\lambda,\varphi,(2q)}_{\mathrm{loc}}(\Omega_{0},\Gamma_{0})
\;\Rightarrow\;
(u,v)\in
\mathcal{D}^{s-\lambda+1,\varphi,(2q)}_{\mathrm{loc}}(\Omega_{0},\Gamma_{0})
\;\Rightarrow\;\ldots\\
\ldots\;\Rightarrow\;
(u,v)\in\mathcal{D}^{s-1,\varphi,(2q)}_{\mathrm{loc}}(\Omega_{0},\Gamma_{0})
\;\Rightarrow\;
(u,v)\in\mathcal{D}^{s,\varphi,(2q)}_{\mathrm{loc}}(\Omega_{0},\Gamma_{0}).
\end{gather*}
Thus, we have proved Theorem \ref{2th4} in the general case.
\end{proof}

\section{Applications}\label{sec7}

In this section, we will apply Theorem~\ref{2th4} to obtain new sufficient conditions under which a generalized solution to the elliptic problem \eqref{2f1}, \eqref{2f2} has continuous classical derivatives of a prescribed order, specifically, is a classical solution. To this end, we also use the following version \cite[Theorems 2.8 and 3.4]{MikhailetsMurach10, MikhailetsMurach14} of the H\"ormander embedding theorem \cite[Theorem 2.2.7]{Hermander63}.

\begin{proposition}\label{2pr4}
Let an integer $\ell\geq0$ and function $\varphi\in\mathcal{M}$ be arbitrarily chosen. Then each of the embeddings $H^{\ell+n/2,\varphi}(\Omega)\subset C^{\ell}(\overline{\Omega})$ and
$H^{\ell+(n-1)/2,\varphi}(\Gamma)\subset C^{\ell}(\Gamma)$ is equivalent to the condition
\begin{equation}\label{2f34}
\int\limits_{1}^{\infty}\frac{dt}{t\,\varphi^{2}(t)}<\infty.
\end{equation}
These embeddings are compact.
\end{proposition}

In connection with this proposition, we recall that, by the Sobolev embedding theorem, the embedding $H^{\sigma}(\Omega)\hookrightarrow C^{\ell}(\overline{\Omega})$ is equivalent to the condition $\sigma>\ell+n/2$. An analogous result is true for the Sobolev spaces over $\Gamma$. Thus, Proposition \ref{2pr4} refines, so to say, the Sobolev embedding theorem in the limiting case.

Applying Theorem \ref{2th4} and Proposition \ref{2pr4}, we can obtain the following result.

\begin{theorem}\label{2th5}
Let an integer $\ell$ meet the inequalities $\ell\geq0$ and $\ell>2q-(n+1)/2$. Suppose that the condition of Theorem~$\ref{2th4}$ is fulfilled for $s=\ell+n/2$ and a certain function $\varphi\in \mathcal{M}$ that satisfies \eqref{2f34}. Then $u\in C^{\ell}(\Omega_{0}\cup\Gamma_{0})$.
\end{theorem}

\begin{proof}
Let a point $x_{0}\in\Omega_0\cup\Gamma_{0}$ be arbitrary. We choose a function $\chi\in C^{\infty}(\overline{\Omega})$ such that  $\mathrm{supp}\,\chi\subset\Omega_0\cup\Gamma_{0}$ and $\chi=1$ in a certain neighbourhood $V_{0}$ of the point $x_{0}$. (Of course, the neighbourhood is considered in the topology of $\overline{\Omega}$.)  According to Theorem~$\ref{2th4}$, we have the inclusion $\chi u\in H^{\ell+n/2,\varphi,(2q)}(\Omega)$. Applying Proposition \ref{2pr4}, we deduce that
$$
\chi u\in H^{\ell+n/2,\varphi,(2q)}(\Omega)=H^{\ell+n/2,\varphi}(\Omega)
\subset C^{\ell}(\overline{\Omega}).
$$
Here, the equality holds true because $\ell+n/2>2q-1/2$ by the condition of this theorem. Therefore, $u\in C^{\ell}(V_{0})$. This means the inclusion $u\in C^{\ell}(\Omega_{0}\cup\Gamma_{0})$ in view of arbitrariness of the choice of $x_{0}\in\Omega_0\cup\Gamma_{0}$.
\end{proof}

\begin{remark}\label{2rem1}
The conclusion of Theorem \ref{2th5} remains valid if we omit the condition $\ell>2q-(n+1)/2$ and additionally suppose that $u\in H^{\sigma}(\Omega)$ for certain $\sigma>2q-1/2$. Indeed, if $\ell\leq2q-(n+1)/2$, then $H^{\sigma}(\Omega)\subset C^{\ell}(\overline{\Omega})$ by virtue of the Sobolev embedding theorem and the assumption $\sigma>2q-1/2$.
\end{remark}

Let us formulate a version of Theorem \ref{2th5} for the function $v_{k}$, with $k\in\{1,\ldots,\varkappa\}$.

\begin{theorem}\label{2th6}
Let $\ell\in\mathbb{Z}$ and $\ell\geq0$. Suppose that the condition of Theorem~$\ref{2th4}$ is fulfilled for $s=\ell-r_{k}+n/2$ and a certain function $\varphi\in \mathcal{M}$ that satisfies \eqref{2f34}. Then $v_{k}\in C^{\ell}(\Gamma_{0})$.
\end{theorem}

\begin{proof}
Let a point $x_{0}\in\Gamma_{0}$ be arbitrary, and let a function $\chi$ and neighbourhood $V_{0}$ be those as in the proof of Theorem~\ref{2th5}. According to Theorem~$\ref{2th4}$, we have the inclusion $\chi v_{k}\in H^{\ell+(n-1)/2,\varphi}(\Gamma)$. Hence, $\chi v_{k}\in C^{\ell}(\Gamma)$ by Proposition~\ref{2pr4}. Therefore, $v_{k}\in C^{\ell}(V_{0}\cap\Gamma)$. This means the inclusion $v_{k}\in C^{\ell}(\Gamma_{0})$ in view of arbitrariness of the choice of $x_{0}\in\Gamma_{0}$.
\end{proof}

\begin{remark}\label{2rem2}
It follows from Proposition~\ref{2pr4} that the condition \eqref{2f34} not only is sufficient in Theorems \ref{2th5} and \ref{2th6} but also is necessary on the class of all the solutions considered.
\end{remark}

Using Theorems \ref{2th5} and \ref{2th6} we can deduce the following sufficient condition under which a generalized solution $(u,v)$ to the elliptic problem \eqref{2f1}, \eqref{2f2} is classical, i.e., $u\in C^{2q}(\Omega)\bigcap C^{m}(\overline{\Omega})$ and $v_{k}\in C^{m+r_{k}}(\Gamma)$ for all $k\in\{1,\ldots,\varkappa\}$. Here, we denote $m:=\nobreak\mathrm{max}\{m_{1},...,m_{q+\varkappa}\}$.

\begin{theorem}\label{2th7}
Let $(u,v)\in\mathcal{D}^{-\infty,(2q)}(\Omega,\Gamma)$, and let $u\in H^{\sigma}(\Omega)$ for certain $\sigma>2q-1/2$. Suppose that the vector
$(u,v)$ is a solution to the elliptic problem \eqref{2f1}, \eqref{2f2} whose right-hand sides satisfy the condition
$$
(f,g)\in\mathcal{E}^{n/2,\varphi_{1}}_{\mathrm{loc}}(\Omega,\varnothing)\cap
\mathcal{E}^{m+n/2-2q,\varphi_{2}}(\Omega)
$$
for some functions $\varphi_{1},\varphi_{2}\in\mathcal{M}$ that meet the condition \eqref{2f34} with $\varphi:=\varphi_{1}$ and $\varphi:=\varphi_{2}$ respectively. Then the solution $(u,v)$ is classical.
\end{theorem}

\begin{proof}
The inclusion $u\in C^{2q}(\Omega)$ follows from the condition $(f,g)\in\mathcal{E}^{n/2,\varphi_{1}}_{\mathrm{loc}}(\Omega,\varnothing)$ by virtue of Theorem~\ref{2th5} with $\ell:=2q$, $\Omega_{0}:=\Omega$, and $\Gamma_{0}:=\varnothing$. Besides, the inclusion $u\in C^{m}(\overline{\Omega})$ follows from the condition $(f,g)\in\mathcal{E}^{m+n/2-2q,\varphi_{2}}(\Omega)$ in view of the same theorem with $\ell:=m$, $\Omega_{0}:=\Omega$, and $\Gamma_{0}:=\Gamma$ if we take into account Remark \ref{2rem1}. Finally, the inclusion $v_{k}\in C^{m+r_{k}}(\Gamma)$ for every $k\in\{1,\ldots,\varkappa\}$ follows from the latter condition by virtue of Theorem \ref{2th6} with $\ell:=m+r_{k}$ and $\Gamma_{0}:=\Gamma$.
\end{proof}

\end{document}